\pdfoutput=1
 \documentclass[11 pt, oneside, reqno]{amsart}
\usepackage[top=1.5in,bottom=1in,left=1in,right=1in]{geometry}
\usepackage{amsmath,amssymb,amsthm,amsopn,extarrows,accents,faktor,enumitem,stmaryrd,mathtools}  
\usepackage[linktocpage=true]{hyperref}
\hypersetup{
    colorlinks = true,
    allcolors = {blue}}
\usepackage[dvipsnames]{xcolor}

\usepackage{setspace}
\usepackage{tikz-cd,tikz}
\usetikzlibrary{matrix,arrows,decorations.pathmorphing}

\newtheorem{lemma}[equation]{Lemma}
\newtheorem{theorem}[equation]{Theorem}
\newtheorem{proposition}[equation]{Proposition}
\newtheorem{corollary}[equation]{Corollary}
\newtheorem*{theorem*}{Theorem}
\theoremstyle{definition}
\newtheorem{definition}[equation]{Definition}
\newtheorem*{remark*}{Remark}
\newtheorem{remark}[equation]{Remark}
\newtheorem*{example}{Example}

\numberwithin{equation}{section}

\DeclareMathOperator{\Hess}{Hess}
\DeclareMathOperator{\codim}{codim}
\DeclareMathOperator{\gr}{\mathbf{Gr}_{\check{G}}}
\DeclareMathOperator{\grt}{\mathbf{Gr}_{\check{T}}}
\DeclareMathOperator{\spec}{Spec}
\DeclareMathOperator{\proj}{Proj}
\DeclareMathOperator{\rees}{Rees}
\DeclareMathOperator{\pt}{pt}
\DeclareMathOperator{\Ad}{Ad}

\DeclareMathOperator{\stdcirc}{\emph{H}^{\,\circ}}

\DeclareMathOperator{\Hstdcirc}{\mathcal{H}^\circ}
\DeclareMathOperator{\Hstdss}{\mathcal{H}^{\text{s}}}

\renewcommand{\sslash}{\mathbin{/\mkern-5mu/}}
\newcommand{\Tlog}{T^*_{D}\overline{G}}
\newcommand{\Zbar}{\overline{\mathcal{Z}}}
\newcommand{\Gbar}{\overline{G}}
\newcommand{\Tbar}{\overline{T}}
\newcommand{\mubar}{\overline{\mu}}
\newcommand{\mux}{\mu_{\mathcal{X}}}
\newcommand{\muh}{\mu_{H}}
\newcommand{\mubarr}{\overline{\mu}_R}
\newcommand{\nubar}{\overline{\nu}}
\newcommand{\nubarr}{\overline{\nu}_R}
\newcommand{\ombar}{\overline{\omega}}
\newcommand{\pibar}{\overline{\pi}}
\newcommand{\Sl}{\mathcal{S}}
\newcommand{\Z}{\mathcal{Z}}
\newcommand{\He}{\mathcal{H}}
\newcommand{\X}{\mathcal{X}}

\newcommand{\hooklongrightarrow}{\lhook\joinrel\longrightarrow}

\title{The partial compactification of the universal centralizer}
\author{Ana B\u{a}libanu}
\address{Department of Mathematics, Harvard University, 1 Oxford St., Cambridge, MA  02138}
\email{ana@math.harvard.edu}
\date{}

\begin{document}
\maketitle

\begin{abstract}
The universal centralizer of a semisimple algebraic group $G$ is the family of centralizers of regular elements, parametrized by their conjugacy classes. When $G$ is of adjoint type, we construct a smooth, log-symplectic fiberwise compactification $\Zbar$ of the universal centralizer $\Z$ by taking the closure of each fiber in the wonderful compactification $\Gbar$. We use the geometry of the wonderful compactification to give an explicit description of the symplectic leaves of $\Zbar$. We also show that its compactified centralizer fibers are isomorphic to certain Hessenberg varieties---we apply this connection to compute the singular cohomology of $\Zbar$, and to study the geometry of the corresponding universal Hessenberg family. 
\end{abstract}

\setcounter{tocdepth}{1}
\tableofcontents

\section*{Introduction}
Let $G$ be a semisimple complex algebraic group of adjoint type. Its Lie algebra $\mathfrak{g}$ contains a transversal slice $\Sl$ for the adjoint action, which was introduced by Kostant \cite{kos:63}. The \emph{universal centralizer} of $G$ is the smooth affine variety
\[\mathcal{Z}\coloneqq \left\{(a,x)\in G\times \Sl\mid a\in G^x\right\}.\]
It can be obtained by Whittaker Hamiltonian reduction from the cotangent bundle of $G$, and this equips it with a natural symplectic structure.

This variety appears as a prominent technical tool in the derived geometric Satake equivalence of Bezrukavnikov and Finkelberg \cite{bez.fin:08}, in the work of Ng\^{o} on the fundamental lemma \cite{ngo:04, ngo:10}, and in the study of Coulomb branches by Braverman, Finkelberg, and Nakajima \cite{bra.fin.nak:18, bra.fin.nak:19}. In particular, Bezrukavnikov, Finkelberg, and Mirkovic \cite{bez.fin.mir:14} have identified $\mathcal{Z}$ with the spectrum of the equivariant homology of the affine Grassmannian of the Langlands dual group $G^\vee$, so $\Z$ is an example of a Coulomb branch as defined by Nakajima \cite{nak:17}.

In this paper we construct a smooth relative compactification of $\mathcal{Z}$. Its fibers are centralizer closures inside the wonderful compactification $\Gbar$, a distinguished equivariant embedding of $G$ first introduced by de Concini and Procesi \cite{dec.pro:83} which encodes the asymptotic behavior of the group ``at infinity.'' We prove the following theorem, as Theorem \ref{bigmain}:

\begin{theorem*}
The relative compactification
\[\Zbar\coloneqq \left\{(a,x)\in\Gbar\times\Sl\mid a\in\overline{G^x}\right\}\]
is a smooth algebraic variety with a natural log-symplectic Poisson structure whose open dense symplectic leaf is $\Z$.
\end{theorem*}

The log-symplectic Poisson structure on $\Zbar$ is inherited from the log-cotangent bundle of $\Gbar$ by Whittaker reduction. Using the geometry of $\Gbar$, we give an explicit algebraic characterization of its symplectic leaves. Moreover, we explain how to obtain $\Zbar$ directly from the Coulomb branch given by \cite{bez.fin.mir:14}. This approach, which involves a Rees construction, uses the realization of $\Gbar$ via the Vinberg monoid and is similar to another example of a compactified Coulomb branch which appears in \cite{bra.fin.nak:19}.

The compactified centralizer fibers of $\Zbar$ have been studied in \cite{bal:16}, where they were identified with closures of generic centralizer orbits on the flag variety $\mathcal{B}$. We refine these results by studying them as a family. In Theorem \ref{isom} we show that they are in fact isomorphic to certain subvarieties of $\mathcal{B}$ known as Hessenberg varieties, which were first introduced by De Mari, Procesi, and Shayman \cite{dem.pro.sha:92}.

\begin{theorem*}
There is an isomorphism of Poisson varieties between the relative compactification $\Zbar$ and the Hessenberg family
\[\He\coloneqq \left\{(\mathfrak{b},x)\in\mathcal{B}\times\Sl\mid \mathfrak{b}\in\Hess(x)\right\}.\]
For any regular element $x\in\mathfrak{g}$, this gives an isomorphism between the compactified centralizer $\overline{G^x}$ and the standard Hessenberg variety $\Hess(x)$. 
\end{theorem*}

Hessenberg varieties have many applications to combinatorics and representation theory, notably in the recent proof of the Shareshian--Wachs conjecture \cite{sha.wac:16} by Brosnan and Chow \cite{bro.cho:18}, in the description of the quantum cohomology of flag varieties in the work of Kostant and of Rietsch \cite{kos:96, rie:01, rie:03}, and in the study of affine Springer fibers by Goresky, Kottwitz, and MacPherson \cite{gor.kot.mac:06}. We apply the theorem above and the geometry of nilpotent Hessenberg varieties to compute the singular cohomology of $\Zbar$, by constructing an explicit affine paving derived from Schubert cells. 

In the other direction, we use the connection between $\Zbar$ and the Hessenberg family $\He$ to show that the natural Poisson structure on the universal family of standard Hessenberg varieties is log-symplectic. We also identify more general, non-regular Hessenberg varieties with moment map fibers of the log-cotangent bundle of $\Gbar$, proving a conjecture of Crooks and R\"oser \cite{cro.ros:20} which was formulated based on an earlier version of the present paper. This gives a natural way to embed any standard Hessenberg variety into the wonderful compactification.

%
%
%
%
%
%
%
\subsection*{Outline}
In Section \ref{symplectic} we review the construction of the universal centralizer $\mathcal{Z}$ as a Whittaker reduction of the cotangent bundle of $G$. In Section \ref{wonderfuls} we recall some facts about the remarkable geometry of $\Gbar$. In Section \ref{third} we obtain $\Zbar$ as a Whittaker reduction of the log-cotangent bundle of $\Gbar$, show that its induced Poisson structure is log-symplectic, and describe its symplectic leaves. 

In Section \ref{hessenberg} we place the relative compactification $\Zbar$ into the setting of Hessenberg varieties by identifying it with the restriction of the standard universal Hessenberg family to the Kostant slice, and use this to compute its singular cohomology. In Section \ref{fifth} we apply this connection to show that the Poisson structure on the standard universal Hessenberg family is log-symplectic. Lastly, in Section \ref{coulomb} we show how the relative compactification $\Zbar$ is related to the realization of the universal centralizer as a Coulomb branch.

%
%
%
%
%
%
%
\subsection*{Acknowledgements}
The author would like to thank Victor Ginzburg, Sam Evens, Ioan M\u{a}rcu\c{t}, Sergei Sagatov, and Travis Schedler for many interesting discussions. Part of this work was completed while the author was supported by a National Science Foundation MSPRF under award DMS--1902921.\\

%
%
%
%
%
%
%
\section{The universal centralizer}
\label{symplectic}
Let $G$ be a connected complex semisimple algebraic group of adjoint type, and let $\mathfrak{g}$ be its Lie algebra. Given an element $x\in\mathfrak{g}$, we write
\[G^x\coloneqq\{g\in G\mid \Ad_gx=x\} \quad\text{and}\quad\mathfrak{g}^x\coloneqq\{y\in \mathfrak{g}\mid \text{ad}_yx=0\}.\]
Fix a principal $\mathfrak{sl}_2$-triple $\{e,h,f\}$---a triple of regular elements in $\mathfrak{g}$ that satisfy the $\mathfrak{sl}_2$ commutation relations. The corresponding \emph{Kostant slice} 
\[\Sl\coloneqq f+\mathfrak{g}^e\] 
consists entirely of regular elements and intersects each regular adjoint orbit in $\mathfrak{g}$ exactly once and transversally. The restriction of the adjoint quotient map
\begin{equation}
\label{adjquot}
\mathfrak{g}\longrightarrow\mathfrak{g}\sslash G
\end{equation}
to the subvariety $\Sl$ is an isomorphism, and therefore $\Sl$ provides a regular section for the adjoint quotient \cite{kos:63}.

\begin{definition}
The \emph{universal centralizer} of $G$ is the affine variety
\[\mathcal{Z}\coloneqq \left\{(a,x)\in G\times\Sl\mid a\in G^x\right\}.\]
\end{definition}

The variety $\Z$ is the family of centralizers of regular elements of $\mathfrak{g}$, parametrized by representatives of their conjugacy classes. Because $G$ is of adjoint type, the fibers of $\mathcal{Z}$ over points in $\Sl$ are connected. Moreover, $\Z$ has a natural symplectic structure, which is inherited from the cotangent bundle of $G$. We recall how to construct this structure using Whittaker reduction, following Kostant \cite{kos:79}. 

Trivializing the cotangent bundle of $G$ with respect to the left $G$-action, and using the Killing form to identify $\mathfrak{g}$ with its dual, we obtain an identification
\begin{equation}
\label{identification}
T^*G\cong G\times\mathfrak{g}.
\end{equation}
Under this isomorphism, the cotangent lifts of left- and right-multiplication on $G$ correspond to the $G\times G$-action
\[(g,h)\cdot (a,x)=(gah^{-1}, \Ad_hx)\qquad\text{for } g,h\in G,\, (a,x)\in G\times\mathfrak{g}.\]
Under the identification
\begin{align*}
\mathfrak{g}\times\mathfrak{g}&\xlongrightarrow{\sim} \mathfrak{g}^*\times\mathfrak{g}^*\\
			(x,y)&\longmapsto ((x,\cdot),-(y,\cdot))
			\end{align*}
given by the Killing form, the associated moment map is
\begin{align}
\label{eq:first}
\mu:T^*G\cong G\times\mathfrak{g}			&\longrightarrow\mathfrak{g}\times\mathfrak{g}\\
			(a,x)					&\longmapsto(\Ad_ax,x).\nonumber
\end{align}
Note that the image of $\mu$ is the set of conjugate pairs inside $\mathfrak{g}\times\mathfrak{g}$, and that the fiber above the diagonal point $(x, x)$ is
\[\mu^{-1}(x, x)\cong G^ x.\]

Let $B$ be the unique Borel subgroup of $G$ whose Lie algebra contains the regular nilpotent $e$, let $N$ be its unipotent radical, and write $\mathfrak{b}$ and $\mathfrak{n}$ for the corresponding Lie algebras. Letting the maximal unipotent subgroup $N\times N$ act on $T^*G$, and using the Killing form to identify $\mathfrak{n}^*\cong\mathfrak{g}/\mathfrak{b}$, the moment map
\[\nu:\,T^*G\longrightarrow\mathfrak{n}^*\times\mathfrak{n}^*\]
factors through the moment map for the $G\times G$-action: 
\begin{equation*}
\label{killing}
\begin{tikzcd}[row sep=large, column sep=large]
T^*G\cong G\times\mathfrak{g}		\arrow{r}{\mu}\arrow{rd}[swap]{\nu}		&\mathfrak{g}\times\mathfrak{g}\arrow{d}\\
  														&\mathfrak{g}/\mathfrak{b}\times\mathfrak{g}/\mathfrak{b}.
\end{tikzcd}
\end{equation*}
The coset $(f, f)\in\mathfrak{g}/\mathfrak{b}\times\mathfrak{g}/\mathfrak{b}$, which corresponds to a regular character in $\mathfrak{n}^*\times\mathfrak{n}^*$, is fixed by the action of $N\times N$. The fiber of $\nu$ above this point is
\begin{align*}
\nu^{-1}(f, f)	&=\left\{(a, x)\in G\times\mathfrak{g}\mid  x\in f+\mathfrak{b}, \,\Ad_ax\in f+\mathfrak{b}\right\}.
\end{align*}

We will repeatedly use the essential fact \cite[Theorem 1.2]{kos:78} that the action map induces an isomorphism 
\begin{equation}
\label{freeness}
N\times(f+\mathfrak{g}^e)\longrightarrow f+\mathfrak{b}.
\end{equation}
In view of this, $N\times N$ acts freely on $\nu^{-1}(f, f)$. Since $\nu$ is a moment map, this implies that $(f, f)$ is a regular value of $\nu$ and therefore that the variety $\nu^{-1}(f, f)$ is smooth. Moreover, there is a natural $N\times N$-equivariant isomorphism
\begin{align*}
N\times N\times \mathcal{Z}		&\longrightarrow \nu^{-1}(f, f)\\
		(n_1,n_2,(a, x))		&\longmapsto (n_1a n_2^{-1}, \Ad_{n_2}x).
\end{align*}
This gives an identification
\[\mathcal{Z}\cong\nu^{-1}(f, f)/(N\times N),\]
and the right-hand side, being a Hamiltonian reduction, is equipped with the structure of a symplectic variety.

\begin{remark}
The symplectic structure on the universal centralizer can also be obtained by considering the diagonal $G$-action on the open dense locus
\[G\times\mathfrak{g}^\text{r}\subset T^*G\]
and taking a Hamiltonian reduction at $0$ along the corresponding moment map
\begin{align*}
\Phi:G\times\mathfrak{g}^\text{r}&\longrightarrow \mathfrak{g}\\
			(a,x)&\longmapsto \Ad_ax-x.
			\end{align*}
The fiber of this moment map above $0$ is
\[\Phi^{-1}(0)=\{(a,x)\in G\times\mathfrak{g}^{\text{r}}\mid a\in G^x\},\]
and its quotient by the diagonal action of $G$ is an abelian group scheme over $\mathfrak{g}\sslash G$. The universal centralizer $\Z$ is then the pullback of this quotient along the isomorphism $\Sl\xlongrightarrow{\sim}\mathfrak{g}\sslash G$. 
\end{remark}

%
%
%
%
%
%
%
\section{The wonderful compactification}
\label{wonderfuls}
To develop a compactified version of the construction outlined in the previous section, we will use the \emph{wonderful compactification} $\Gbar$ of the group $G$. This is a canonical, smooth, equivariant projective embedding of $G$ which was introduced in \cite{dec.pro:83}. A detailed survey of its geometry can be found in \cite{eve.jon:08}, and here we collect only the facts that we will need later. 

\subsection{Geometry of the wonderful compactification}
\label{bundle}
Let $l$ be the rank of $G$. The $G\times G$-equivariant compactification $\Gbar$ contains $G$ as an open dense subset, and the complement
\[D\coloneqq \Gbar\backslash G\]
is a divisor with simple normal crossings and irreducible components $D_1,\ldots,D_l$. The group $G\times G$ acts on $D$ with finitely many orbits, which are in bijection with the proper subsets $I\subset\{1,\ldots,l\}$. The closure of the orbit $\mathcal{O}_I$ is
\[\overline{\mathcal{O}_I}=\bigcap_{i\not\in I} D_i\]
---the intersection of the divisor components not indexed by $I$. In particular, the $G\times G$-orbit closures are smooth, and there is a unique closed orbit of minimal dimension obtained by intersecting all the irreducible components of $D$.

Let $T=G^h$ be the centralizer of the regular semisimple element $h$. Then $T$ is a maximal torus contained in the Borel $B$, and we write 
\[\Delta=\{\alpha_1,\ldots,\alpha_l\}\]
for the corresponding set of simple roots. Let $P_I$ be the standard parabolic subgroup generated by $B$ and the negative simple root groups indexed by $\{\alpha_i\mid i\in I\}$, and let $P_I^-$ be the opposite parabolic with respect to $T$. Write $U_I$ and $U_I^-$ for their respective unipotent radicals, $L_I=P_I\cap P_I^-$ for their shared Levi component, and $Z(L_I)$ for its center. We will denote by $\mathfrak{p}_I^\pm, \mathfrak{u}_I^\pm,$ and $\mathfrak{l}_I$ the Lie algebras of these subgroups.

For each index set $I$ there is a distinguished basepoint $z_I\in\mathcal{O}_I$ whose stabilizer is
\[\text{Stab}_{G\times G}(z_I)=\left\{(us,vt)\in P_I\times P_I^-\mid u\in U_I, v\in U_I^-, s,t\in L_I, st^{-1}\in Z(L_I)\right\}.\]
This implies that the orbit $\mathcal{O}_I$ fibers over the product $G/P_I\times G/P_I^-$ with fiber isomorphic to $L_I/Z(L_I)$, which is a semisimple group of adjoint type. Taking the closure, we obtain a smooth fibration
\begin{equation*}
\begin{tikzcd}[row sep=large]
\overline{L_I/Z(L_I)}		\arrow[r, hook]&\overline{\mathcal{O}_I}\arrow[d, twoheadrightarrow]\\
  														&G/P_I\times G/P_I^-,
\end{tikzcd}
\end{equation*}
where the fiber is the wonderful compactification of the adjoint group $L_I/Z(L_I)$. In particular, the unique closed orbit of minimal dimension is isomorphic to the product of full flag varieties
\[G/B\times G/B^-.\]
%

%
%
%
%
%
%
%
\subsection{The logarithmic cotangent bundle}
Because the divisor $D$ has simple normal crossings, the sheaf of vector fields on $\Gbar$ which are tangent to $D$ is locally free. The associated vector bundle $T_D\Gbar$ is called the \emph{logarithmic tangent bundle} of $\Gbar$. Because $D$ is $G\times G$-stable, the group action induces a map of vector bundles
\[\Gbar\times\mathfrak{g}\times\mathfrak{g}\longrightarrow T_D\Gbar,\]
which is surjective by \cite[Example 2.5]{bri:09}. The kernel of this map is Lagrangian, and therefore it is isomorphic to the \emph{logarithmic cotangent bundle} $\Tlog$---the bundle whose sections are differential forms with logarithmic poles along $D$. This gives an equivariant embedding 
\begin{equation}
\label{loghom}
\Tlog\hooklongrightarrow \Gbar\times\mathfrak{g}\times\mathfrak{g},
\end{equation}
of vector bundles over $\Gbar$, where $G\times G$ acts on the right-hand side by
\[(g,h)\cdot (a,y,x)=(gah^{-1}, \Ad_gy, \Ad_hx)\qquad\text{for } g,h\in G,\, (a,y,x)\in G\times\mathfrak{g}\times\mathfrak{g}.\]

\begin{remark} 
Under our earlier identification \eqref{identification}, the map \eqref{loghom} extends the natural embedding
\begin{align}
\label{openemb}
T^*G\cong G\times\mathfrak{g}&\hooklongrightarrow G\times\mathfrak{g}\times\mathfrak{g} \\
				(a,x)&\longmapsto (a, \Ad_ax,x).\nonumber
				\end{align}
The compactification $\Gbar$ can also be realized \cite[Section 3.2]{eve.jon:08} as a subvariety of the Grassmannian 
\[Gr(\dim\mathfrak{g},\mathfrak{g}\times\mathfrak{g}),\]
by taking the closure of the image of the embedding
\begin{align*}
G&\longrightarrow Gr(\dim\mathfrak{g},\mathfrak{g}\times\mathfrak{g})\\
			a&\longmapsto a\cdot\mathfrak{g}_\Delta,
			\end{align*}
where 
\[a\cdot\mathfrak{g}_\Delta=\{(\Ad_ax,x)\mid x\in\mathfrak{g}\}.\]
Using \eqref{openemb}, one sees that the log-cotangent bundle $\Tlog$ is precisely the restriction of the tautological bundle on the Grassmannian to this subvariety. It then follows from \cite[Remark 3.9]{eve.jon:08} that, under the embedding \eqref{loghom}, the fiber of $\Tlog$ above the orbit basepoint $z_I\in\mathcal{O}_I$ is identified with the subalgebra
\begin{equation}
\label{basefiber}
T^*_{D,z_I}\Gbar\cong\mathfrak{p}_I\times_{\mathfrak{l}_I}\mathfrak{p}_I^-.
\end{equation}
\end{remark}

Consider the $G\times G$-equivariant morphism 
\begin{align}
\label{compactified mu}
\overline{\mu}:\,\,&\Tlog\longrightarrow\mathfrak{g}\times\mathfrak{g}\\
				&(a, y, x)\longmapsto( y, x)\nonumber
				\end{align}
given by projection onto the fibers of $\Gbar\times\mathfrak{g}\times\mathfrak{g}$. Because it extends the usual moment map $\mu$ defined in \eqref{eq:first}, the map $\mubar$ is called a \emph{compactified moment map} in \cite{bri:09} and \cite{kno:94}. We conclude this section by listing some properties of this morphism.

\begin{lemma}
\label{image mu}
The image of $\overline{\mu}$ is the variety $\mathfrak{g}\times_{\mathfrak{g}\sslash G}\mathfrak{g}$, which consists of pairs of elements in $\mathfrak{g}\times\mathfrak{g}$ that lie in the closure of the same adjoint orbit.
\end{lemma}
\begin{proof}
Since the fibers of the adjoint quotient are closures of regular adjoint orbits \cite[Theorem 0.7]{kos:63}, the variety $\mathfrak{g}\times_{\mathfrak{g}\sslash G}\mathfrak{g}$ consists of pairs of elements in the same orbit closure. Moreover, since this variety is the preimage of the diagonal under the adjoint quotient map of $\mathfrak{g}\times\mathfrak{g}$, it is closed. 

The map $\overline{\mu}$ is proper because it factors through the projection
\begin{equation*}
\begin{tikzcd}[row sep=large, column sep=large]
\Tlog		\arrow[hookrightarrow]{r}\arrow{rd}[swap]{\overline{\mu}}		&\Gbar\times\mathfrak{g}\times\mathfrak{g}\arrow{d}\\
  														&\mathfrak{g}\times\mathfrak{g}.
\end{tikzcd}
\end{equation*}
Because it is proper, its image is closed, so it is the closure of the image of $\mu$. Since the image of $\mu$ is the set of conjugate pairs in $\mathfrak{g}\times\mathfrak{g}$, and since a function on $\mathfrak{g}\times\mathfrak{g}$ vanishes on the locus of conjugate pairs if and only if it vanishes on $\mathfrak{g}\times_{\mathfrak{g}\sslash G}\mathfrak{g}$, its closure is $\mathfrak{g}\times_{\mathfrak{g}\sslash G}\mathfrak{g}$.
\end{proof}

\begin{lemma}
\label{normal mu}
The variety $\mathfrak{g}\times_{\mathfrak{g}\sslash G}\mathfrak{g}$ is normal.
\end{lemma}
\begin{proof}
Since $\mathfrak{g}\times_{\mathfrak{g}\sslash G}\mathfrak{g}$ is the image of $\mubar$, it is irreducible, and therefore it has dimension 
\[2\dim G-\dim(\mathfrak{g}\sslash G)=2\dim G-l.\]
Let 
\[f_1,\ldots,f_l\in\mathbb{C}[\mathfrak{g}]^G\]
be a minimal set of generators for the algebra of $G$-invariant polynomials on $\mathfrak{g}$. Then
\[\mathfrak{g}\times_{\mathfrak{g}\sslash G}\mathfrak{g}=\{(x,y)\in\mathfrak{g}\times\mathfrak{g}\mid f_i(x)-f_i(y)=0 \text{ for all }i=1,\ldots,l\}\]
is the vanishing locus of $l$ algebraically independent polynomials, so it is a complete intersection.

Because the differentials $df_1,\ldots df_l$ are linearly independent at every point of the regular locus $\mathfrak{g}^{\text{r}}$ \cite[Claim 6.7.10]{chr.gin:97}, the subset
\[\mathfrak{g}^{\text{r}}\times_{\mathfrak{g}\sslash G}\mathfrak{g}^{\text{r}}=\{(x,y)\in\mathfrak{g}^{\text{r}}\times\mathfrak{g}^{\text{r}}\mid x\in G\cdot y\}\subset\mathfrak{g}^{\text{r}}\times\mathfrak{g}^{\text{r}}\]
is a smooth open subvariety of $\mathfrak{g}\times_{\mathfrak{g}\sslash G}\mathfrak{g}$. Moreover, its complement has codimension at least $2$ because adjoint orbits are even-dimensional \cite[Claim 6.7.10]{chr.gin:97}. It follows that $\mathfrak{g}\times_{\mathfrak{g}\sslash G}\mathfrak{g}$ is a complete intersection with no codimension-one singularities, so it is normal. \qedhere
\end{proof}

\begin{lemma} 
\label{connected}
The fibers of $\overline{\mu}$ are connected.
\end{lemma}
\begin{proof} 
The fiber of $\overline{\mu}$ above a regular semisimple pair is a connected toric variety \cite[Example 2.5]{bri:09}. Moreover, by Lemmas \ref{image mu} and \ref{normal mu}, the image of $\overline{\mu}$ is normal. Since $\overline{\mu}$ is proper with normal image and the fiber above a generic point is connected, it follows by Stein factorization that every fiber of $\overline{\mu}$ is connected.\qedhere\\
\end{proof}
%

%
%
%
%
%
%
%
\section{The relative compactification of $\Z$}
\label{third}
In this section we compactify the fibers of the universal centralizer $\Z$ by taking their closures inside the wonderful compactification $\Gbar$, and we describe how the symplectic structure on $\Z$ extends to the boundary of this new variety. We begin by recalling some standard facts about log-symplectic manifolds, and we refer to \cite[Section 5]{pym:18} for a more general introduction. 

\subsection{Log-symplectic structures}
Suppose that $X$ is a complex manifold containing a divisor 
\[Z=\cup Z_i\]
with simple normal crossings. A \emph{log-symplectic} structure on $(X,Z)$ is a closed, nondegenerate $2$-form with logarithmic poles along $Z$. Such a form restricts to a symplectic form on $X\backslash Z$, and induces an isomorphism
\[T_{Z}X\xlongrightarrow{\sim}T^*_{Z}X\]
between the log-tangent bundle and the log-cotangent bundle of $X$. 

There is a corresponding isomorphism of associated sheaves
\[\mathcal{T}_{X,Z}\xlongrightarrow{\sim}\mathcal{T}^*_{X,Z}\]
between the log-tangent sheaf, whose sections are vector fields parallel to $Z$, and the log-cotangent sheaf, whose sections are differential forms with logarithmic poles along $Z$. Reversing this isomorphism and restricting to the subsheaf of differential forms without poles, we get a morphism
\[T^*X\longrightarrow TX.\]
This bundle map corresponds to a generically non-degenerate Poisson bivector. In other words, a log-symplectic structure is a type of Poisson structure.

\begin{example}
\label{logcot}
Suppose that $M$ is a complex manifold with a simple normal crossing divisor $D$. Let $X\coloneqq T^*_DM$, let
\[\varphi:X=T^*_DM\longrightarrow M\]
be the bundle map, and let $Z\coloneqq \varphi^{-1}(D)$ be the simple normal crossing divisor in $T^*_DM$ obtained by taking the preimage of $D$. 

The dual of the differential of $\varphi$, which we write as
\[\overline{\lambda}:\,T^*_DM\longrightarrow T^*_ZX,\]
is a logarithmic $1$-form on $X$ analogous to the usual Liouville $1$-form. Its differential 
\[\ombar\coloneqq d\overline{\lambda}\]
is a closed, nondegenerate logarithmic $2$-form. In this way, the log-cotangent bundle of any complex manifold with a simple normal crossing divisor has a canonical log-symplectic structure.

Explicitly, if $a_1,\ldots,a_n$ are local coordinates on $M$ such that the divisor $D$ is cut out by the vanishing of the product $a_1\cdot\ldots\cdot a_k$, and if $ x_1,\ldots, x_n$ are induced coordinates on the fiber, the log-symplectic form is given locally by
\begin{equation}
\label{incoords}
\ombar=\sum_{i=1}^k\frac{da_i}{a_i}\wedge d x_i+\sum_{i=k+1}^nda_i\wedge d x_i.
\end{equation}
The associated Poisson bivector is
\begin{equation}
\label{incoordspi}
\pibar=\sum_{i=1}^k{a_i}\partial_{a_i}\wedge \partial_{x_i}+\sum_{i=k+1}^n\partial_{a_i}\wedge \partial_{x_i}.
\end{equation}
In particular, the restriction of $\ombar$ to the cotangent bundle $T^*(X\backslash Z)\subset T^*_ZX$ is the canonical symplectic form.
\end{example}

For completeness, we include the following log-symplectic analogue of the Marsden-Weinstein-Meyer Theorem. Although it appears to be known to experts, we were unable to find a reference in the existing literature. A similar statement is discussed in \cite[Section 7.7]{gua.li.pel.rat:17}.

\begin{proposition}
\label{log-reduction}
Let $(X,Z)$ be a log-symplectic Poisson manifold with log-symplectic form $\ombar$. Suppose that $H$ is a connected Lie group with a Hamiltonian action on $X$, let 
\[\rho:X\longrightarrow\mathfrak{h}^*,\]
be the moment map, and let $ x\in\mathfrak{h}^*$ be a point fixed by the coadjoint $H$-action. Suppose that $H$ acts on the fiber $\widetilde{X}\coloneqq \rho^{-1}( x)$ freely and that the intersection $\widetilde{X}\cap Z$ is a divisor in $\widetilde{X}$. Then
\begin{enumerate}
\item $X_0\coloneqq \widetilde{X}/H$ is smooth and $Z_0\coloneqq (\widetilde{X}\cap Z)/H$ is a simple normal crossing divisor;
\item there is a log-symplectic structure $\ombar_0$ on $(X_0,Z_0)$ such that 
\begin{equation}
\label{pullbacks}
\iota^*\ombar=q^*\ombar_0,
\end{equation}
where $\iota:\widetilde{X}\hooklongrightarrow X$ is the inclusion and $q:\widetilde{X}\longrightarrow X_0$ is the quotient map.
\end{enumerate}
\end{proposition}

\begin{proof}
(1) Let $L\subset X$ be any symplectic leaf that intersects the fiber $\widetilde{X}$---since the action of $H$ is Hamiltonian and $H$ is connected, $L$ is $H$-stable, and we can consider the restricted moment map
\[\rho_{\vert L}:L\longrightarrow \mathfrak{h}^*.\]
The group $H$ acts freely on $\rho_{\vert L}^{-1}( x)$ and $L$ is symplectic, so $ x$ is a regular value of $\rho_{\vert L}$. Being a regular value of $\rho_{\vert L}$ for every symplectic leaf $L$, $x$ is also a regular value of $\rho$. 

This implies that $\widetilde{X}$, and therefore $X_0$, are both smooth. Moreover, writing 
\[Z=Z_1\cup\ldots\cup Z_l\]
as a union of irreducible components, we obtain a stratification of $Z$ indexed by the subsets $I\subset\{1,\ldots,l\}$ which is given by the smooth submanifolds
\[Z^I=(\cap_{i\in I}Z_i)\backslash (\cup_{i\not\in I}Z_i).\]
Any such stratum $S$ is an $H$-stable union of symplectic leaves, and $x$ is a regular value of $\rho_{\vert S}$, so $\widetilde{X}\cap S$ is smooth. Therefore $\widetilde{Z}=\widetilde{X}\cap Z$ and $Z_0$ are simple normal crossing divisors. 

(2) Since the action of $H$ is Hamiltonian, it preserves the Poisson structure on $X$ and therefore also the degeneracy divisor $Z$. This induces an infinitesimal action map
\begin{align*}
\mathfrak{h}&\longrightarrow \Gamma(T_ZX)\\
x&\longmapsto\vartheta_x
\end{align*}
which assigns to each Lie algebra element a vector field tangent to $Z$.

The restriction of $\ombar$ to $\widetilde{X}$ is a degenerate logarithmic $2$-form whose kernel is generated by these Hamiltonian vector fields. The quotient morphism $q$ induces a surjection of logarithmic tangent bundles
\[q_*:T_{\widetilde{Z}}\widetilde{X}\longrightarrow T_{Z_0}X_0\]
whose kernel is exactly $\langle \vartheta_x\mid x\in\mathfrak{g}\rangle$. Therefore, the $2$-form $\ombar_{\vert\widetilde{X}}$ descends to a non-degenerate logarithmic $2$-form $\ombar_0$ on $(X_0,Z_0)$, which satisfies \eqref{pullbacks} by definition.
\end{proof}

%
%
%
%
%
%
%
\subsection{Whittaker reduction of $\Tlog$}
We now define a fiberwise compactification of the universal centralizer by
\[\Zbar\coloneqq \left\{(a, x)\in\Gbar\times\Sl\mid a\in\overline{G^ x}\right\}.\]
We will use a log-symplectic Hamiltonian reduction of $\Tlog$ to prove the following theorem:

\begin{theorem}
\label{bigmain}
The relative compactification $\Zbar$ is a smooth algebraic variety whose boundary is a divisor with simple normal crossings. The variety $\Zbar$ has a natural log-symplectic structure which extends the symplectic structure on $\Z$.
\end{theorem}
\begin{proof}
The log-cotangent bundle $\Tlog$ is equipped with the canonical log-symplectic form $\ombar$ defined in Example \ref{logcot}, which restricts to the usual symplectic form on $T^*G$. Therefore the action of $G\times G$ on $\Tlog$ is Hamiltonian, and the corresponding moment map is precisely the compactified moment map $\mubar$ defined in \eqref{compactified mu}.

We follow the same reduction procedure as in Section \ref{bundle}. Restricting to the action of the maximal unipotent subgroup $N\times N$, we obtain a moment map
\[\nubar:\,\Tlog\longrightarrow\mathfrak{n}^*\times\mathfrak{n}^*.\]
Making once again the identification $\mathfrak{n}^*\cong\mathfrak{g}/\mathfrak{b}$ gives a commutative diagram
\begin{equation*}
\label{mu factors}
\begin{tikzcd}[row sep=large, column sep=large]
\Tlog		\arrow{r}{\overline{\mu}}\arrow{rd}[swap]{\nubar}		&\mathfrak{g}\times\mathfrak{g}\arrow{d}\\
  													&\mathfrak{g}/\mathfrak{b}\times\mathfrak{g}/\mathfrak{b}.
\end{tikzcd}
\end{equation*}
The points in the fiber of $\nubar$ above $(f,f)$ are of the form 
\[(a, x_1, x_2)\in\Gbar\times\mathfrak{g}\times\mathfrak{g}\]
with $ x_1, x_2\in f+\mathfrak{b}$. By \eqref{freeness}, this implies that $N\times N$ acts freely on $\nubar^{-1}(f,f)$. In particular, $(f,f)$ is a regular value of $\nubar$ and the fiber $\nubar^{-1}(f,f)$ is smooth.

By Lemma \ref{connected}, $\nubar^{-1}(f,f)$ is connected. Being connected and smooth, it is irreducible. It follows that the intersection 
\[\nubar^{-1}(f,f)\cap T^*G=\nu^{-1}(f,f)\]
is dense in $\nubar^{-1}(f,f)$. But this intersection is precisely
\[\nu^{-1}(f,f)=\left\{(g,\,n_1  x,\,n_2 x)\in\,\Tlog\mid n_1,n_2\in N, \, x\in f+\mathfrak{g}^e, \,g\in n_1G^ x n_2^{-1}\right\},\]
and therefore 
\[\nubar^{-1}(f,f)=\left\{(a,\,n_1  x,\,n_2 x)\in\,\Tlog\mid n_1,n_2\in N, \, x\in f+\mathfrak{g}^e, \,a\in n_1\overline{G^ x} n_2^{-1}\right\},\]
This gives an isomorphism
\[\Zbar\cong \nubar^{-1}(f,f)/N\times N.\]

By Proposition \ref{log-reduction}, $\Zbar$ is a smooth variety and its boundary is a divisor with simple normal crossings. Moreover, $\Zbar$ has a natural log-symplectic structure. Its open dense symplectic leaf is the reduction of the open dense leaf $T^*G$ of $\Tlog$---in other words, it is the universal centralizer $\Z$.
\end{proof}

As part of the proof of Theorem \ref{bigmain}, we showed that the fibers of $\mubar$ above regular elements of $\mathfrak{g}\times\mathfrak{g}$ are precisely the closures of the fibers of the usual moment map $\mu$. In particular, we have the following immediate corollary.

\begin{corollary}
Let $ x\in\mathfrak{g}$ be a regular element. There is an isomorphism
\[\mubar^{-1}(x,x)\cong\overline{G^ x},\]
where the right-hand side is the closure of the centralizer of $x$ in the wonderful compactification.
\end{corollary}

\begin{remark}
The compactified centralizer fibers of $\Zbar$ are studied in \cite{bal:16}, where they are identified with the closures of generic centralizer orbits on the flag variety $G/B$. This result will be extended in Section \ref{hessenberg}.

The fiber above a regular semisimple element is the projective toric variety whose fan is the fan of Weyl chambers \cite[Remark 4.5]{eve.jon:08}. The fiber above the principal nilpotent $f$ is the Peterson variety \cite[Theorem 4.1]{bal:16}.  In particular, away from the semisimple locus, the fibers of $\Zbar$ are singular and generally not normal \cite[Theorem 14]{kos:96}. 
\end{remark}

%
%
%
%
%
%
%
\subsection{Symplectic leaves}
The symplectic leaves of $\Zbar$ are Hamiltonian reductions of the symplectic leaves of $\Tlog$. Therefore, in order to describe the former, we first give a general description of the latter. In view of \eqref{incoords}, the restriction of the logarithmic cotangent bundle to each $G\times G$-orbit $\mathcal{O}_I$ is a union of symplectic leaves. We will give an explicit characterization of the leaves in each such restriction.

Fix a subset $I\subseteq\{1,\ldots,l\}$. Let
\[\mathfrak{c}_I:\mathfrak{p}_I\longrightarrow\mathfrak{p}_I/[\mathfrak{p}_I,\mathfrak{p}_I]\]
be the quotient of the parabolic subalgebra $\mathfrak{p}_I$ by its derived subgroup. There is a direct-sum decomposition
\[\mathfrak{p}_I=[\mathfrak{p}_I,\mathfrak{p}_I]\oplus Z(\mathfrak{l}_I),\]
and the fibers of $\mathfrak{c}_I$ are the points in $\mathfrak{p}_I$ with the same component in the center $Z(\mathfrak{l}_I)$ of the Levi. Applying \eqref{basefiber}, we first give a criterion for when two points in the fiber
\[T^*_{D,z_I}\Gbar\cong \mathfrak{p}_I\times_{\mathfrak{l}_I}\mathfrak{p}_I^-\]
are in the same symplectic leaf.

\begin{proposition}
\label{fiberleaf}
Let $(x,x^-),(y,y^-)\in\mathfrak{p}_I\times_{\mathfrak{l}_I}\mathfrak{p}_I^-$. The points $(z_I,x,x^-)$ and $(z_I,y,y^-)$ are in the same symplectic leaf of $\Tlog$ if and only if
\[\mathfrak{c}_I(x)=\mathfrak{c}_I(y).\]
\end{proposition}
\begin{proof}
Let $n=\dim N$. Around the basepoint $z_I\in\mathcal{O}_I$ there are algebraic local coordinates 
\[\{a_1^\pm,\ldots,a_n^\pm,z_1,\ldots,z_l\}\]
such that $D$ is cut out by the vanishing of the product $\Pi_{i\in I}z_i$ and such that the induced coordinates on the fiber
\[\{x_1^\pm,\ldots,x_n^\pm,\zeta_1,\ldots,\zeta_l\}\]
have the property that the directions $\{\zeta_i\mid i\in I\}$ are tangent to the diagonal copy of $Z(l_I)$ \cite[Section 2.2]{eve.jon:08}. In these coodinates, the log-symplectic Poisson bivector defined in \eqref{incoordspi} is 
\[\pibar=\sum_{j}\frac{\partial}{\partial a^\pm_j}\wedge \frac{\partial}{\partial x^\pm_j}+\sum_{i\not\in I}\frac{\partial}{\partial z_i}\wedge \frac{\partial}{\partial \zeta_i}+\sum_{i\in I}z_i\frac{\partial}{\partial z_i}\wedge\frac{\partial}{\partial \zeta_i}.\]
This bivector is tangent to $\{z_i=\zeta_i=0 \mid i\in I\}$ and is nondegenerate along this submanifold. Therefore, two points of 
\[T^*_{D,z_I}\Gbar\cong \mathfrak{p}_I\times_{\mathfrak{l}_I}\mathfrak{p}_I^-\]
are in the same symplectic leaf if and only if they have the same component in the center $Z(\mathfrak{l}_I)$ of the Levi. 
\end{proof}

Recall that the orbit $\mathcal{O}_I$ fibers over the product $G/P_I\times G/P_I^-$ of partial flag varieties. Therefore we can associate to each $a\in\Gbar$ a pair of parabolic subalgebras
\[(\mathfrak{p}_a,\mathfrak{q}_a)\in G/P_I\times G/P_I^-.\]
If $a=gz_Ih^{-1}$, these parabolics satisfy $\mathfrak{p}_a=g\cdot\mathfrak{p}_I$ and $\mathfrak{q}_a=h\cdot\mathfrak{p}_I^-$. There is a canonical identification of algebras
\[\mathfrak{p}_a/[\mathfrak{p}_a,\mathfrak{p}_a]\cong\mathfrak{p}_I/[\mathfrak{p}_I,\mathfrak{p}_I].\]
We define the abelian algebra
\[\mathfrak{a}_I\coloneqq \mathfrak{p}_I/[\mathfrak{p}_I,\mathfrak{p}_I]\]
to be the ``universal'' Cartan algebra associated to the partial flag variety $G/P_I$, and then for each $a\in\mathcal{O}_I$ we have a corresponding quotient map
\[\mathfrak{c}_a:\mathfrak{p}_a\longrightarrow\mathfrak{a}_I.\]

Let $T^*_{D,I}\Gbar$ be the restriction of $\Tlog$ to the orbit $\mathcal{O}_I$. Since $\Tlog$ is an equivariant vector bundle, there is an isomorphism
\[T^*_{D,I}\Gbar\cong (G\times G)\times_{\text{Stab}_{G\times G}(z_I)}(\mathfrak{p}_I\times_{\mathfrak{l}_I}\mathfrak{p}_I^-).\]
This induces a smooth, well-defined morphism
\begin{align*}
T^*_{D,I}\Gbar&\longrightarrow \mathfrak{a}_I \\
	(a,x,x^-)&\longmapsto\mathfrak{c}_a(x).
	\end{align*}
In the next proposition we show that its fibers are precisely the symplectic leaves of $T^*_{D,I}\Gbar.$

\begin{proposition}
Two points $(a,x,x^-)$ and $(b,y,y^-)$ of $T^*_{D,I}\Gbar$ are in the same symplectic leaf of $\Tlog$ if and only if
\[\mathfrak{c}_a(x)=\mathfrak{c}_b(y).\]
\end{proposition}
\begin{proof}
Under the action of $G\times G$, $(a,x,x^-)$ is conjugate to some point $(z_I,x_0,x_0^-)$ in the fiber over $z_I$ and $(b,y,y^-)$ is conjugate to some $(z_I,y_0,y_0^-)$. Since the symplectic leaves of $\Tlog$ are $G\times G$-stable, $(a,x,x^-)$ and $(b,y,y^-)$ are in the same symplectic leaf if and only if the same is true of $(z_I,x_0,x_0^-)$ and $(z_I,y_0,y_0^-)$. By Proposition \ref{fiberleaf}, this happens if and only if $\mathfrak{c}_I(x_0)=\mathfrak{c}_I(y_0).$ Since $\mathfrak{c}_a(x)=\mathfrak{c}_I(x_0)$ and $\mathfrak{c}_b(y)=\mathfrak{c}_I(y_0),$ the theorem follows.
\end{proof}

The symplectic leaves of the Hamiltonian reduction $\Zbar$ are the Hamiltonian reductions of the symplectic leaves of $\Tlog$. We obtain the following immediate corollary.

\begin{corollary}
Two points $(a, x),(b,y)\in\Zbar$ are in the same symplectic leaf if and only if $a$ and $b$ are contained in in the same $G\times G$-orbit and $\mathfrak{c}_a(x)=\mathfrak{c}_b(y).$\\
\end{corollary}

%
%
%
%
%
%
%
\section{The universal family of Hessenberg varieties}
\label{hessenberg}
In this section we will show that each fiber of $\Zbar$ is isomorphic to a certain type of subvariety of $G/B$ known as a Hessenberg variety. This connection makes it possible to compute the singular cohomology of $\Zbar$, and to characterize more general fibers of the moment map $\mubar$. 

\subsection{Hessenberg varieties}
We keep the notation of the previous sections, and we begin with some general background. A \emph{Hessenberg subspace} of $\mathfrak{g}$ is a $B$-submodule of $\mathfrak{g}$ that contains $\mathfrak{b}$. Let $H$ be such a space, and consider the associated vector bundle $G\times_B H.$ It has a canonical Poisson structure, observed also in a special case in \cite{abe.cro:18}, which comes from a Hamiltonian reduction as follows. The right action of $G$ on $T^*G$, given by
\[h\cdot(g,x)=(gh^{-1}, \Ad_hx),\quad\text{for }h\in G,\, (g,x)\in G\times\mathfrak{g},\]
is Hamiltonian with moment map
\begin{align*}
\mu_R:T^*G\cong G\times\mathfrak{g}&\longrightarrow \mathfrak{g}\\
					(g,x)&\longmapsto x
					\end{align*}
Restricting to the action of the Borel $B$ and applying the Killing form to identify $\mathfrak{b}^*\cong\mathfrak{g}/\mathfrak{n}$, we get a moment map
\[T^*G\longrightarrow\mathfrak{g}/\mathfrak{n},\]
and we consider the preimage of the $B$-stable subset $H/\mathfrak{n}\subset\mathfrak{g}/\mathfrak{n}$. Reducing, we obtain the smooth Poisson variety
\[G\times_B H=\mu_{R}^{-1}(H)/B.\]
Its symplectic leaves are in bijective correspondence with the $B$-orbits on the quotient $H/\mathfrak{n}$, and it inherits a Hamiltonian action of $G$ on the left, with moment map 
\begin{align*}
\muh:\, G\times_B H&\longrightarrow\mathfrak{g} \\
		[g:x]&\longmapsto \Ad_gx.
		\end{align*}

The \emph{Hessenberg variety} associated to the point $x\in\mathfrak{g}$ is the moment map fiber 
\[\Hess(x):\,=\muh^{-1}(x)=\left\{gB\in G/B \mid \Ad_{g^{-1}}x\in H\right\}.\]
Because of this, the bundle $G\times_BH$ is called the \emph{universal family of Hessenberg varieties}. We will consider its restriction
\[\mathcal{H}\coloneqq \left\{(gB, x)\in G/B\times\Sl\mid gB\in\Hess(x)\right\}\]
to the Kostant slice $\Sl$. 

\begin{proposition}
The space $\mathcal{H}$ is a smooth Poisson variety and a proper flat family over $\Sl$ of relative dimension $\dim (H/\mathfrak{b})$.
\end{proposition}
\begin{proof}
We will take a Whittaker reduction of the space $G\times_B H$. Restricting to the action of $N$, we get a moment map 
\[\nu_{H}:G\times_B H\longrightarrow \mathfrak{n}^*.\]
Under the usual Killing form identification, it factors as
\begin{equation*}
\begin{tikzcd}[row sep=large, column sep=large]
G\times_B H		\arrow{r}{\muh}\arrow[rd, swap, "\nu_{H}"]	&\mathfrak{g}\arrow{d}\\
  														&\mathfrak{g}/\mathfrak{b}.
\end{tikzcd}
\end{equation*}
As before, we take the fiber above the coset of $f$ in $\mathfrak{g}/\mathfrak{b}$:
\begin{align*}
\nu_{H}^{-1}(f)&=\muh^{-1}(f+\mathfrak{b}) \\
			&=\left\{(gB, x)\in G/B\times\mathfrak{g} \mid x\in(f+\mathfrak{b}), gB\in\Hess(x)\right\}.
						\end{align*}
By \eqref{freeness}, the action of $N$ on this fiber is free, so $\nu^{{-1}}_\He(f)$ is smooth. The action map
\begin{align*}
N\times \mathcal{H}&\longrightarrow\nu^{-1}_H(f) \\
(n, (gB,x))&\longmapsto (ngB, \Ad_nx)
\end{align*}
is an isomorphism, and this realizes $\mathcal{H}$ as the smooth Poisson variety
\[\mathcal{H}\cong\nu^{{-1}}_H(f)/N.\]

Let
\[\pi_\He:\He\longrightarrow\Sl\]
be the structure morphism. Since it factors through the projection $G/B\times\Sl\longrightarrow\Sl$, it is proper. Moreover, the fibers of $\pi_\He$ are regular Hessenberg varieties, which all have dimension $\dim (H/\mathfrak{b})$ \cite[Corollary 2.7]{pre:17}. Since $\mathcal{H}$ and $\Sl$ are both smooth and the fibers of $\pi_\He$ are equidimensional, this morphism is flat. 
\end{proof}

We end this section by showing that the flat family $\He$ has a contracting $\mathbb{C}^*$-actions which allows us to compute its cohomology directly from the cohomology of the fiber above the regular nilpotent point. Recall that $\{e,h,f\}$ is our fixed principal $\mathfrak{sl}_2$-triple, and choose a one-parameter subgroup
\begin{align*}
\gamma:\mathbb{C}^*\longrightarrow G 
		\end{align*}
whose Lie algebra is spanned by the regular semisimple element $h$. The spaces $\Sl$ and $\He$ have natural $\mathbb{C}^*$-actions given by
\begin{align*}
&t\cdot x=t^{2}\Ad_{\gamma(t)}x,\qquad\qquad\qquad\qquad\quad\,\text{for }t\in\mathbb{C}^*, x\in\Sl,\\
&t\cdot (gB,x)=(\gamma(t)gB,\, t^{2}\Ad_{\gamma(t)}x),\quad\quad\,\text{for }t\in\mathbb{C}^*, (gB,x)\in\He,\text{ and}
\end{align*}
The structure morphism $\pi_\He$ is $\mathbb{C}^*$-equivariant with respect to these actions. Moreover, both of these actions are contracting as $t\to 0$---in the limit, the principal slice $\Sl$ is contracted to the principal nilpotent $f\in \Sl$ and the variety $\He$ is contracted to the $\mathbb{C}^*$-fixed points of $\Hess(f)$. 

\begin{proposition}
\label{betti}
There are isomorphisms of singular cohomology rings 
\[H^*(\mathcal{H},\mathbb{C})\cong H^*(\Hess(f),\mathbb{C}).\]
\end{proposition}
\begin{proof}
This follows from a standard fact in the topology of algebraic varieties. Suppose
\[f:X\longrightarrow S\]
is a proper $\mathbb{C}^*$-equivariant morphism between smooth varieties, and suppose that the $\mathbb{C}^*$-action contracts the base $S$ to the point $o\in S$. Then the action gives a deformation retraction of the family $X$ onto a small Euclidean neighborhood $U$ of $f^{-1}(o)$. This implies that 
\[H^*(X,\mathbb{C})\cong H^*(U,\mathbb{C})\cong H^*(f^{-1}(o),\mathbb{C}).\qedhere\]
\end{proof}

%
%
%
%
%
%
%
\subsection{The standard Hessenberg space}
From now on we let $H$ be the \emph{standard Hessenberg space}
\[H\coloneqq \left(\sum_{\alpha\in\Delta}\mathfrak{g}_{-\alpha}\right)\oplus\mathfrak{b},\]
which is the sum of the positive Borel subalgebra and the negative simple root spaces. We write $\He$ for the corresponding universal family of regular Hessenberg varieties.

When $s\in\mathfrak{g}$ is regular and semisimple, the associated Hessenberg variety $\Hess(s)$ is the complete toric variety whose fan is the fan of Weyl chambers \cite[Theorem 11]{dem.pro.sha:92}. When $f\in\mathfrak{g}$ is regular and nilpotent, the regular Hessenberg variety $\Hess(f)$ is precisely the Peterson variety. 

\begin{remark}
When $s\in\mathfrak{g}$ is semisimple, the variety $\Hess(s)$ is connected \cite{pre:15}. Then the moment map
\[\muh:G\times_B H\longrightarrow\mathfrak{g}\]
is a proper map whose target is smooth and whose generic fibers are connected, so by Stein factorization all of its fibers are connected. It follows that the standard Hessenberg variety $\Hess(x)$ is connected for every element $x\in\mathfrak{g}$.
\end{remark}

We now relate the standard family of Hessenberg varieties to the compactified universal centralizer of the previous section. First we will need the following two elementary lemmas.

\begin{lemma}
\label{trivial centralizers}
Let $x\in\Sl$. Then $G^x\cap B=1$.
\end{lemma}
\begin{proof}
Write $x=f+v\in f+\mathfrak{g}^e$ and suppose that $g\in G^x\cap B$. Let $g=tu$ with $t\in T$ and $u\in N$. The Lie algebra $\mathfrak{g}$ is graded by eigenvalues for the adjoint action of the regular semisimple $h$. The regular nilpotent $f$ sits in degree $-2$, and the Borel $\mathfrak{b}$ is the sum of the non-negative eigenspaces. Then
\[f+v=\Ad_g(f+v)=\Ad_{tu}f+\Ad_gv=\Ad_tf+\text{(higher degree terms)}.\]
It follows that $t\in G^f$, and therefore $t=1$ and $g\in N$. But by \eqref{freeness}, $N$ acts freely on $f+\mathfrak{b}$, so $g=1$.
\end{proof}

Write $W$ for the Weyl group associated to the maximal torus $T$. For any element $w\in W$, we abuse notation to denote by $w\in N_G(T)$ an arbitrary choice of preimage in the normalizer of $T$.

\begin{lemma}
\label{bruhats}
Let $x\in\Sl$ and suppose that $g\in G$ is such that $\Ad_gx\in\mathfrak{b}.$ Then $g$ is contained in the maximal Bruhat cell $Bw_0N$.
\end{lemma}
\begin{proof}
Let $x=f+v\in f+\mathfrak{g}^e$ and use the Bruhat decomposition to write $g=bwu$ for some $b\in B$, $u\in N$, and $w\in N_G(T)$. Since 
\[\Ad_b(\Ad_{wu}(f+v))\in\mathfrak{b},\]
it follows that $\Ad_{wu}(f+v)$ is an element of $\mathfrak{b}.$ Writing
\[\Ad_u(f+v)=f+v'\in f+\mathfrak{b},\]
we obtain 
\[\Ad_{w}(f+v')\in\mathfrak{b}.\]
Since, as a sum of root vectors, $f$ has a nonzero component in each negative simple root space \cite[Lemma 3.2.12]{chr.gin:97}, this implies that $w$ flips the sign of every simple root. Therefore $w=w_0$ is the longest element of the Weyl group.
\end{proof}

Because the Kostant slice $\Sl$ is contained in the standard Hessenberg space $H$, we can define the morphism
\begin{align}
\label{morph}
\alpha:\quad\mathcal{Z}\,\,\,&\longrightarrow \He \\
		(g,x)&\longmapsto (gB,x),\nonumber
		\end{align}
which is compatible with the structure maps over the slice $\mathcal{S}$. If we equip $\Zbar$ with the $\mathbb{C}^*$-action
\[t\cdot (a,x)=(\gamma(t)a\gamma(t)^{-1},\, t^{2}\Ad_{\gamma(t)}x),\quad\text{for }t\in\mathbb{C}^*, (a,x)\in\Zbar,\]
then the open dense subset $\mathcal{Z}$ is $\mathbb{C}^*$-stable and the morphism $\alpha$ is $\mathbb{C}^*$-equivariant. Moreover, the action map gives an isomorphism
\[B\times\Sl\longrightarrow H^\circ\coloneqq \left(\sum_{\alpha\in\Delta}\mathfrak{g}_{-\alpha}\backslash\{0\}\right)\oplus\mathfrak{b}.\]
Therefore, the image of $\alpha$ is the open subset of $\He$ given by
\begin{align*}
\He^\circ&\coloneqq \{(gB,x)\in G/B\times\Sl\mid g\in G^x\}\\
			&\,= \{(gB,x)\in G/B\times\Sl\mid \Ad_{g^{-1}}(x)\in H^\circ\}=\He\cap (G\times_BH^\circ) 					.
					\end{align*}

As before, write 
\[\pi_{\Zbar}:\Zbar\longrightarrow \Sl\quad\text{and}\quad \pi_{\He}:\He\longrightarrow\Sl\]
for the structure maps and let $\Sl^{\text{s}}\subset \Sl$ be the semisimple locus of the principal slice. Consider the two open dense subsets
\[\Zbar^{\text{s}}=\pi_{\Zbar}^{-1}(\Sl^{\text{s}})\subset\Zbar\quad\text{and}\quad\Hstdss=\pi_{\He}^{-1}(\Sl^{\text{s}})\subset\He.\]
Above a semisimple element of $\Sl$, the fiber of $\Zbar$ and the fiber of $\He$ are both isomorphic to the toric variety corresponding to the fan of Weyl chambers. Therefore $\alpha$ extends to an isomorphism along every semisimple fiber.

\begin{lemma}
\label{extension}
The restriction of the map $\alpha$ to $\mathcal{S}^\text{s}$ extends to an isomorphism of varieties
\[\Zbar^{\text{s}}\longrightarrow \Hstdss.\]
\end{lemma}
\begin{proof}
The restriction of the adjoint quotient map \eqref{adjquot} to the regular locus $\mathfrak{t}^{\text{r}}$ of the maximal Cartan is the $W$-cover
\[\mathfrak{t}^{\text{r}}\longrightarrow\mathfrak{t}^{\text{r}}/W\cong\mathcal{S}^{\text{s}}.\]
The pullback of $\Z$ to $\mathfrak{t}^{\text{r}}$ along this map is the trivial group scheme $\mathfrak{t}^{\text{r}}\times T,$ and the pullback of $\Zbar$ to $\mathfrak{t}^{\text{r}}$ is the trivial bundle $\mathfrak{t}^{\text{r}}\times \Tbar.$ The pullback of $\He$ to $\mathfrak{t}^{\text{r}}$ is the variety 
\[\He_{\mathfrak{t}^{\text{r}}}\coloneqq\left\{(x,gB)\in \mathfrak{t}^{\text{r}}\times G/B \mid \Ad_{g^{-1}}x\in H\right\}.\]
In particular, the constant section of $\He$ given by the positive Borel $B$ pulls back to a smooth section
\[x\longmapsto (x,g_xB)\]
of the family $\He_{\mathfrak{t}^{\text{r}}}$ and, since each point $(x,g_xB)$ is $G$-conjugate to a point $(s,B)\in\mathcal{H}$, Lemma \ref{bruhats} implies that $g_x$ is contained in the maximal Bruhat cell $Bw_0N$. Lastly, the map $\alpha$ pulls back to the morphism
\begin{align*}
\tilde{\alpha}: \mathfrak{t}^{\text{r}}\times T&\longrightarrow \He_{\mathfrak{t}^{\text{r}}} \\
							(x,t)&\longmapsto (x,tg_xB).
							\end{align*}
Since $\tilde{\alpha}$ is the pullback of $\alpha$ through a $W$-invariant map, it is $W$-equivariant. To prove the lemma, it is then sufficient to show that $\tilde{\alpha}$ extends to a morphism defined on all of $\mathfrak{t}^{\text{r}}\times\Tbar$. This morphism, being $W$-equivariant, will then descend through the covering map to an extension of $\alpha$ defined on all of $\Zbar^{\text{s}}$, and any such extension is an isomorphism because it is an isomorphism on each fiber.

Consider the basepoint $z_{\varnothing}\in\Tbar$ of the closed $G\times G$-orbit on $\Gbar$. If $(x_h,t_h)$ is a 1-parameter family in $\mathfrak{t}^{\text{r}}\times T$ which approaches $(x,z_{\varnothing})$ as $h$ tends to infinity, it follows from \cite[Section 2.4]{eve.jon:08} that
\[\lim_{h\to\infty}\alpha_i(t_h)\to\infty\]
for every simple root $\alpha_i$. Since the maximal Schubert cell is precisely the attracting set of the point $B\in G/B$ under the action of $T$ as the values of the simple roots tend to infinity \cite[Theorem 3.1.9]{chr.gin:97}, 
\[\lim_{h\to\infty}\tilde{\alpha}(x_h,t_h)=\lim_{h\to\infty}(x_h, t_hg_{x_h}B)=(x,B).\]
Since this limit is independent of the chosen one-parameter family, it follows that the map $\tilde{\alpha}$ extends continuously to the point $(x,z_{\varnothing})$. Because any $T$-fixed point of $\Tbar$ is $W$-conjugate to $z_{\varnothing}$, this implies that $\tilde{\alpha}$ extends continuously to all $T$-fixed points of $\mathfrak{t}^{\text{r}}\times\Tbar$.

The set $X\subset \mathfrak{t}^{\text{r}}\times\Tbar$ of points to which the morphism $\tilde{\alpha}$ fails to extend continuously consists exactly of those points above which the closure of the graph of $\tilde{\alpha}$ has fibers of nonzero dimension (see, for example, the proof of \cite[Lemma 7.4]{deo.han:20}). Since fiber dimension is upper semi-continuous, $X$ is a closed subvariety of $\mathfrak{t}^{\text{r}}\times \Tbar$. Since $\tilde{\alpha}$ is $T$-equivariant, $X$ is also $T$-stable. However, by the above discussion $X$ contains no $T$-fixed points---therefore, it must be the empty set. It follows that $\tilde{\alpha}$ extends to a morphism on all of $\mathfrak{t}^{\text{r}}\times\Tbar$, and therefore that $\alpha$ extends to an isomorphism as desired.
\end{proof}

\begin{theorem}
\label{isom}
The map $\alpha$ extends to an isomorphism
\[\overline{\alpha}: \Zbar\longrightarrow \He.\]
\end{theorem}
\begin{proof}
The fiber of $\Zbar$ above the regular nilpotent element $f\in\Sl$ is the closure of the unipotent subgroup $G^f$ in the wonderful compactification $\Gbar$. By \cite[Theorem 4.1]{bal:16}, the restriction of $\alpha$ to the fiber above $f$ extends to an isomorphism
\[\overline{G^f}\longrightarrow\Hess(f).\]

Let $\mathcal{L}$ be an $\mathbb{C}^*$-equivariant ample line bundle on $\He$. The pullback of its restriction to $\Hess(f)$ is then an ample line bundle on $\overline{G^f}$. Since the higher cohomology groups of the structure sheaf of $\Hess(f)$ vanish \cite[Theorem 1.1]{abe.fuj.zen:20}, the deformations of this line bundle are unobstructed and therefore it extends to an ample line bundle on a formal neighborhood of $\overline{G^f}$ in $\Zbar$ \cite[Theorem 3.3.11(iii)]{ser:06}. Pulling back through the contracting $\mathbb{C}^*$-action, it futher extends to an ample line bundle on $\Zbar$. Because $\alpha$ is $\mathbb{C}^*$-equivariant, along $\Z$ this extension agrees with the pullback $\alpha^*\mathcal{L}$. 

By Lemma \ref{extension}, the map $\alpha$ extends to an isomorphism
\[\Z\cup\Zbar^{\text{s}}\xlongrightarrow{\sim}\He^\circ\cup\Hstdss.\]
Because the fibers of $\pi_{\Zbar}$ and $\pi_\He$ are connected, the complements of the open sets $\Z\cup\Zbar^{\text{s}}$ and $\He^\circ\cup\Hstdss$ have codimension at least 2. Since both $\Zbar$ and $\He$ are projective over $\Sl$ and smooth, and since the pullback $\alpha^*\mathcal{L}$ is an ample line bundle, it follows from a Hartog argument \cite{mat.mum:64} (see also \cite[Theorem 11.39]{kol:22}) that $\alpha$ extends to an isomorphism
\[\overline{\alpha}: \Zbar\longrightarrow \He.\qedhere\]
\end{proof}

Since every regular element in $\mathfrak{g}$ is conjugate to some element of the Kostant slice $\Sl$, Theorem \ref{isom} has the following immediate corollary.

\begin{corollary}
\label{general orbits}
For any regular $x\in\mathfrak{g}$, there is a $G^x$-equivariant isomorphism of varieties 
\[\overline{G^x}\cong \Hess(x).\]
\end{corollary}

\begin{remark}
It was proved in \cite{bal:16} that there is an isomorphism between the compactified regular centralizer $\overline{G^x}$ and the closure of a generic $G^x$-orbit on the flag variety $G/B$. It then follows from Corollary \ref{general orbits} that the closure of a generic $G^x$-orbit on $G/B$ is isomorphic to the standard Hessenberg variety associated to $x$. 
\end{remark}

%
%
%
%
%
%
%
\subsection{Cohomology of $\Zbar$}
Proposition \ref{betti} shows that the cohomology of $\He$ is isomorphic to the cohomology of the Peterson variety $\Hess(f)$. We improve this result by using the $\mathbb{C}^*$-action on $\He$ described in the previous section to construct an affine paving. In view of Theorem \ref{isom}, this gives a basis for the singular cohomology space $H^*(\Zbar,\mathbb{C})$.

The $\mathbb{C}^*$-action contracts $\He$ to the $\mathbb{C}^*$-fixed points of the Peterson variety $\Hess(f)$. Because $h$ is regular, these coincide with the $T$-fixed points on the flag variety which lie in $\Hess(f)$, which are known to be exactly
\[\{w_IB\in G/B\mid I\subset\Delta\},\]
where $w_I\in W$ is the longest word of the parabolic Weyl group indexed by the subset of simple roots $I$ \cite[Proposition 5.8]{har.tym:17}.

For each $I\subset\Delta$, consider the attracting set
\[X_I\coloneqq \left\{(gB,x)\in\He\mid gB\in\Hess(x)\cap Bw_IB\right\}.\]
It is shown in \cite[Proposition 6.3]{bal:16} that
\[\dim\,\left(\Hess(f)\cap B w_IB\right)=|I|.\]
Since $\He$ is flat over $\Sl$, it follows that the dimension of the attracting set $X_I$ is $l+\vert I\vert$. 

\begin{proposition}
The attracting sets $X_I$ give a stratification of $\He$ by affine spaces, and the classes 
\[\left\{[X_I]\mid I\subset\Delta\right\}\]
form an additive basis for the singular cohomology $H^*(\He,\mathbb{C})$, where the degree of the class $[X_I]$ is
\[2l-2|I|.\]
\end{proposition}
\begin{proof}
This is a direct consequence of the Bialynicki-Birula decomposition \cite[Theorem 4.3]{bia:73}. We remark that the result of \emph{loc. cit.} is stated for smooth projective schemes. However, the projectivity assumption is needed only to establish that the fixed set of the $\mathbb{C}^*$-action is projective and that, in the limit, every point of the variety flows to a fixed point. If these conditions are satisfied, the Bialynicki-Birula decomposition holds more generally for any smooth quasi-projective variety. See for example \cite[Appendix B]{beh.bry.sze:13}, where an action satisfying these conditions is called \emph{circle compact}.
\end{proof}

%
%
%
%
%
%
%
\subsection{General fibers of the moment map}
Consider the moment map
\[\mubarr:\Tlog\longrightarrow\mathfrak{g}\]
corresponding to the right action of $G$ on the log-cotangent bundle $\Tlog$. Restricting to the action of the maximal unipotent subgroup $N$, we have a moment map
\[\nubarr:\Tlog\longrightarrow\mathfrak{n}^*.\]
Once again identifying $\mathfrak{n}^*\cong\mathfrak{g}/\mathfrak{b}$, we obtain a commutative diagram
\begin{equation*}
\begin{tikzcd}[row sep=large, column sep=large]
\Tlog		\arrow{r}{\mubarr}\arrow[rd, swap, "\nubarr"]		&\mathfrak{g}\arrow{d}\\
  														&\mathfrak{g}/\mathfrak{b}.
\end{tikzcd}
\end{equation*}
As in the previous sections, by \eqref{freeness} the action of $N$ on the fiber $\nubarr^{-1}(f)$ is free, and so this fiber is smooth.

Viewing $\Tlog$ as a subbundle of the trivial bundle $\Gbar\times\mathfrak{g}\times\mathfrak{g}$ as in \eqref{loghom}, define the closed subvariety
\[\X\coloneqq \left\{(a,y,x)\in \Tlog\mid x\in\Sl\right\}.\]
Through the action map we get an isomorphism
\[N\times\X\xlongrightarrow{\sim}\nubarr^{-1}(f),\]
and Proposition \ref{log-reduction} implies that 
\[\X\cong\nubarr^{-1}(f)/N\]
is a smooth log-symplectic variety. The left action of $G$ descends to a Hamiltonian action on $\X$, and the associated moment map is
\begin{align*}
\mux:\quad\X\quad&\longrightarrow\mathfrak{g} \\
	(a,y,x)&\longmapsto y.
	\end{align*}

Let $\gamma:\mathbb{C}^*\longrightarrow G$ once again be the one-parameter subgroup whose Lie algebra is spanned by the regular semisimple element $h$, and define an action of $\mathbb{C}^*$ on $\Gbar\times\mathfrak{g}\times\mathfrak{g}$ by
\begin{equation}
\label{contracts}
t\cdot(a,y,x)=(a\gamma(t)^{-1},t^2y,t^2\Ad_{\gamma(t)}x)\quad\text{for }t\in\mathbb{C}^*, (a,y,x)\in \Gbar\times\mathfrak{g}\times\mathfrak{g}.
\end{equation}
Since this action stabilizes both the closed subvariety $\Gbar\times\mathfrak{g}\times\Sl$ and the locus \eqref{openemb} of points of $T^*G$, it induces a well-defined action on $\X$.

The moment map $\mux$ is then equivariant with respect to the $\mathbb{C}^*$-action on $\mathfrak{g}$ which scales by a factor of $t^2$. In the limit as $t\to0$, this action therefore contracts the variety $\X$ to the fiber of $\mux$ above $0$. This fiber is given by
\begin{align*}
\mux^{-1}(0)&=\{(a,0,x)\in\Tlog\mid x\in\Sl\}\\
		&=\{(gz_\varnothing h^{-1},0,f)\mid \Ad_{h^{-1}}f\in \mathfrak{b}^-\}\\
		&\cong (G\times B^-)\cdot z_\varnothing\\
		&\cong G/B,
		\end{align*}
where $z_\varnothing$ is the basepoint of the closed orbit of minimal dimension in $\Gbar$ \cite[Corollary 4.22]{cro.ros:20}.

Via the embedding \eqref{openemb}, the cotangent bundle $T^*G$ sits inside $\Tlog$ as the locus of points of the form $(a,\Ad_ax,x)$. Intersecting $\X$ with this open locus, we obtain
\[\X^\circ\coloneqq \left\{(a,\Ad_ax,x)\in \Tlog\mid a\in G, x\in\Sl\right\}.\]
This variety, which is the Whittaker reduction of $T^*G$ with respect to the right-action of $G$, is the open dense symplectic leaf of $\X$. It is isomorphic to 
\[G\times_N(f+\mathfrak{b})\cong G\times\Sl\]
---the \emph{twisted cotangent bundle} of the base affine space $G/N$. 

We define a morphism
\begin{align*}
\beta:\qquad\X^\circ\qquad&\longrightarrow G\times_BH \\
		(a,\Ad_ax,x)&\longrightarrow [a:x].
		\end{align*}
This map commutes with the moment maps $\mux$ and $\muh$, and by Lemma \ref{trivial centralizers} it is injective. Moreover, it is $\mathbb{C}^*$-equivariant when the vector bundle $G\times_BH$ with the $\mathbb{C}^*$-action
\[t\cdot[a:x]=[a:t^2x]\quad\text{for }t\in\mathbb{C}^*, \,[a:x]\in G\times_BH\]
given by scaling along the fibers.

\begin{lemma}
\label{betaext}
The morphism $\beta$ extends to an isomorphism between the fiber of $\mux$ above $0$ and the zero-section of the vector bundle $G\times _BH$.
\end{lemma}
\begin{proof}
In the log-cotangent bundle $\Tlog$, the point $(z_\varnothing,0,f)$ can be realized as the limit
\[(z_\varnothing,0,f)=\lim_{t\to0}t\cdot(1,f,f)\]
under the action \eqref{contracts}. Since any point of $\mux^{-1}(0)$ is a left translate of this one, we can continuously extend $\beta$ to this fiber by
\[\beta(gz_\varnothing,0,f)=\lim_{t\to0}t\cdot\beta(g,\Ad_g(f),f)=\lim_{t\to0}t\cdot[g:f]=[g:0].\]
Since the stabilizer of $z_\varnothing$ under the left action of $G$ is exactly the Borel subgroup $B$, it follows that this extension is an isomorphism of varieties.
\end{proof}

We can now use Theorem \ref{isom} to prove the following proposition, which was conjectured in \cite[Conjecture 4.26]{cro.ros:20} based on an earlier version of the present paper.

\begin{proposition}
\label{conj}
The map $\beta$ extends to an isomorphism
\[\overline{\beta}:\X\longrightarrow G\times_BH\]
which commutes with the moment mapts $\mux$ and $\muh$.
\end{proposition}
\begin{proof}
Consider the regular loci
\[\X^{\text{r}}\coloneqq\mux^{-1}(\mathfrak{g}^{\text{r}})=\left\{(a,y,x)\in \Tlog\mid x\in\Sl, y\in\mathfrak{g}^\text{r}\right\}\]
and
\[\muh^{-1}(\mathfrak{g}^{\text{r}})=\left\{(gB,x)\mid x\in\mathfrak{g}^{\text{r}}, gB\in\Hess(x)\right\}=G\times_BH^{\text{r}}.\]
Since two regular elements are in the same adjoint orbit closure if and only if they are conjugate, we have isomorphisms 
\[\Zbar\times_{\Sl}\mathfrak{g}^{\text{r}}\cong \X^{\text{r}}\quad\text{and}\quad \He\times_{\Sl}\mathfrak{g}^{\text{r}}\cong G\times_BH^{\text{r}}\]
Therefore, the isomorphism
\[\overline{\alpha}:\Z\longrightarrow\He\]
of Theorem \ref{isom}
induces an isomorphism
\[\beta^{\text{r}}:\X^{\text{r}}\xlongrightarrow{\sim}G\times_BH^{\text{r}}.\]
which agrees with $\beta$ on the overlap with $\X^\circ$.

We now bound the codimension of the complements of the subvarieties $\X^\text{r}$ and $G\times_BH^{\text{r}}$. First, since regular semisimple adjoint orbits are closed, each irreducible component of the complement of $\X^{\text{r}}$ is strictly contained in an irreducible component of the complement of $\X\cap(\Gbar\times\mathfrak{g}\times\mathfrak{g}^{\text{rs}})$, which has codimension at least $1$. It follows that
\[\codim_{\X}(\X\backslash\X^{\text{r}})\geq2.\]
Moreover, the codimension of the complement of $G\times_BH^{\text{r}}$ in $G\times_BH$ is equal to
\[\codim_{(G\times_BH)}(G\times_B(H\backslash H^\text{r}))=\codim_H(H\backslash H^{\text{r}})\geq 3\]
where the last inequality follows, for instance, from the proof of \cite[Lemma 4.15]{cro.ros:20}. 

Therefore $\beta^{\text{r}}$ is a birational map between open subvarieties whose complements have codimension at least $2$. Moreover, by Lemma \ref{betaext}, $\beta$ extends continuously to an isomorphism between the zero-fiber of $\mux$ and the zero-section $G/B$ of $G\times_BH$. Since any line bundle on $G/B$ extends to a unique line bundle on the vector bundle $G\times_BH$ by pulling back along the bundle map, it follows once again that $\beta^{\text{r}}$ intertwines two ample line bundles. Therefore, once again by \cite[Theorem 11.39]{kol:22} it follows that $\beta^{\text{r}}$ extends to an isomorphism
\[\overline{\beta}:\X\longrightarrow G\times_BH.\qedhere\]
\end{proof}

\begin{corollary}
Let $y\in\mathfrak{g}$ and let $x\in\Sl$ be the unique point in the Kostant slice such that $y$ is in the closure of the adjoint orbit of $x$. Then there is an isomorphism
\[\mubar^{-1}(y,x)\cong\Hess(y).\]
\end{corollary}
\begin{proof}
The isomorphism of Proposition \ref{conj} induces an isomorphism between corresponding moment map fibers 
\[\mux^{-1}(y)\cong\muh^{-1}(y).\]
The left-hand side is precisely $\mubar^{-1}(y,x)$, and the right-hand side is $\Hess(y).$\qedhere\\
\end{proof}

%
%
%
%
%
%
%
\section{The Poisson structure on the standard Hessenberg family}
\label{fifth}
Both sides of the isomorphism 
\[\overline{\alpha}:\Zbar\xlongrightarrow{\sim}\He\]
constructed in Theorem \ref{isom} are equipped with natural Poisson structures. In this section we will show that the morphism $\overline{\alpha}$ is Poisson. We will use this observation to prove that the Poisson structure on the standard universal Hessenberg family $G\times_BH$ is log-symplectic. To do this we will apply the theory of Poisson transversals, which we first briefly recall. See, for instance, \cite[Section 2]{fre.mar:17} for more details.

\subsection{Poisson transversals}
Let $X$ be a complex Poisson manifold. A \emph{Poisson transversal} in $X$ is an embedded submanifold $V$ such that, for any symplectic leaf $(L, \omega_L)$ of $X$ which meets $V$,
\begin{itemize}
\item $V$ is transverse to $L$, and 
\item the pullback of $\omega_L$ to $V\cap L$ is symplectic. 
\end{itemize}
The property of being a Poisson transversal induces a natural Poisson structure on $V$, and the symplectic leaves of this structure are precisely the intersections
\[(V\cap L, \omega_{L\vert V\cap L}).\]

\begin{example}
\begin{enumerate}[leftmargin=20pt]
\item If $X$ is symplectic, then the Poisson transversals in $X$ are the symplectic submanifolds, and the induced Poisson structures are simply the restrictions of the symplectic form.
\item If all the symplectic leaves of $X$ have the same dimension, and if $V$ is a submanifold of complementary dimension which meets each leaf transversally, then $V$ is a Poisson transversal and its induced Poisson structure is trivial.
\item If $X$ is the semisimple Lie algebra $\mathfrak{g}$ with the standard Kostant-Kirillov Poisson structure, and $\{e,h,f\}\subset\mathfrak{g}$ is any $\mathfrak{sl}_2$-triple, the Slodowy slice $f+\mathfrak{g}^e$ is a Poisson transversal \cite[Section 3.2]{gan.gin:02}.
\end{enumerate}
\end{example}

The following lemma appears in \cite{fre.mar:17} in the setting of real differential manifolds. For complex manifolds the proof is identical, so we do not reproduce it.
\begin{lemma}{\cite[Lemma 7]{fre.mar:17}}
\label{transversals}
Let $\varphi:X_1\longrightarrow X_2$ be a Poisson morphism between smooth Poisson varieties, and suppose that $V\subset X_2$ is a Poisson transversal. Then the preimage $\varphi^{-1}(V)$ is a Poisson transversal in $X_1$. In particular, $\varphi^{-1}(V)$ is smooth.
\end{lemma}

\begin{remark}
It follows from Lemma \ref{transversals} that, for any Hamiltonian $G$-manifold $M$ with moment map
\[\varphi:M\longrightarrow\mathfrak{g},\]
the preimage $\Phi^{-1}(\Sl)$ is a Poisson transversal in $M$ and therefore has a natural Poisson structure. Moreover, \eqref{freeness} gives an isomorphism
\begin{equation}
\label{canonic}
\varphi^{-1}(\Sl)\cong \varphi^{-1}(f+\mathfrak{b})/N,
\end{equation}
which is an isomorphism of Poisson manifolds. While this is true in full generality \cite[Proposition 3.13]{cro.ros:21}, in the symplectic and log-symplectic setting it is particularly easy to see and seems to have been known to experts for some time, so we include here a short explanation. The isomorphism \eqref{canonic} fits into the diagram
\begin{equation*}
\begin{tikzcd}[row sep=large, column sep=large]
\varphi^{-1}(\Sl)		\arrow[r, hook, "\imath"]	\arrow[rd,swap,"\sim"]	&\varphi^{-1}(f+\mathfrak{b}) \arrow[d, "q"]\arrow[r,hook,"\jmath"] &M\\
  										&\varphi^{-1}(f+\mathfrak{b})/N.&
\end{tikzcd}
\end{equation*}
Letting $\ombar$ be the $2$-form induced by Whittaker reduction on the quotient $\varphi^{-1}(f+\mathfrak{b})/N$, we have
\[\imath^*q^*\ombar=\imath^*\jmath^*\omega,\]
and therefore \eqref{canonic} is a Poisson map. In particular the Poisson manifolds $\He$, $\Z$, and $\Zbar$ can be realized as the Poisson transversals
\[\muh^{-1}(\Sl),\quad \mu^{-1}(\Sl\times\Sl),\quad\text{and}\quad\mubar^{-1}(\Sl\times\Sl)\]
in $G\times_BH$, $T^*G$, and $\Tlog$ respectively. 
\end{remark}

We prove the following result on log-symplectic Poisson transversals in the setting of complex manifolds. Because in the next section all Poisson algebraic varieties are smooth, we will be able to apply it to our algebraic setting as well.

\begin{proposition}
\label{log-transversals}
Let $(X,\pi)$ be a complex Poisson manifold and suppose that $V\subset X$ is a Poisson transversal which intersects every symplectic leaf. If the induced Poisson structure on $V$ is log-symplectic, then the Poisson structure on $X$ is also log-symplectic.
\end{proposition}
\begin{proof}
Choose a point $x\in V$. By the Weinstein splitting theorem \cite[Theorem 1.25]{lau.pic.van:13}, there is an open neighborhood $U\subset X$ containing $x$ such that $U$ is Poisson-diffeomorphic to the product of $V\cap U$ and a symplectic manifold $(L,\rho_L)$---that is,
\[(U,\pi_{\vert U})\cong (U\cap V, \pi_V)\times (L,\rho_L).\]
If $V$ is log-symplectic, then the Poisson structure $\pi_{\vert U}$ is also log-symplectic. It follows that there is an open neighborhood of $V$ in $X$ where the Poisson bivector is log-symplectic at every point.

If $V$ intersects every symplectic leaf, then any point in $X$ is reached from a point of $V$ by flowing along Hamiltonian vector fields. Since the Poisson structure is invariant under this flow, it follows that it is log-symplectic at every point of $X$.\qedhere
\end{proof}

%
%
%
%
%
%
%
\subsection{The log-symplectic structure on $G\times_BH$}
The symplectic leaves of $G\times_B H$ correspond bijectively to the $B$-orbits on the quotient space $H/\mathfrak{n}$. They are described in detail in \cite{abe.cro:18}. In particular, there is a unique open dense orbit whose preimage in $H$ is $\stdcirc$. It follows that the unique open dense symplectic leaf of $G\times_BH$ is
\[G\times_B\stdcirc.\]

\begin{lemma}
The open dense symplectic leaf of $\He$ is precisely the image $\He^\circ$ of the morphism $\alpha$ defined in \eqref{morph}.
\end{lemma}
\begin{proof}
The open dense symplectic leaf of $\He$ is the Whittaker reduction of $G\times_B\stdcirc$, which we can write as
\[L_0\coloneqq \left\{(gB,x)\in G/B\times\Sl\mid \Ad_{g^{-1}}x\in \stdcirc\right\}.\]
Recall that $\alpha$ was defined by
\begin{align*}
\alpha:\, &\mathcal{Z}\longrightarrow \He \\
		&(g,x)\longmapsto (gB,x),
		\end{align*}
and that the points of $\mathcal{Z}$ are pairs $(g,x)\in G\times\mathfrak{g}$ with the property that $g$ centralizers $x$. It follows immediately that $\He^\circ\subset L_0$. Therefore, it is enough to check that the fiber of $L_0$ over $x\in\Sl$ is a single $G^x$-orbit. 

Let $g_1B, g_2B$ be elements of this fiber, so that $\Ad_{g_1^{-1}}x,\,\Ad_{g_2^{-1}}x\in \stdcirc$. By \eqref{freeness}, the action map gives an isomorphism
\[B\times \Sl\longrightarrow \stdcirc.\]
This implies that any two elements in $\stdcirc$ that are $G$-conjugate are actually $B$-conjugate, so there is an element $b\in B$ with 
\[\Ad_{g_1^{-1}}x=\Ad_{bg_2^{-1}}x.\]
It follows that $g_1bg_2^{-1}\in G^x$, and therefore the flags $g_1B$ and $g_2B$ are in the same $G^x$-orbit.
\end{proof}

\begin{proposition}
\label{log-isom}
The isomorphism
\[\overline{\alpha}:\Zbar\longrightarrow\He\]
defined in Theorem \ref{isom} is an isomorphism of Poisson manifolds. In particular, the Poisson structure on $\He$ is log-symplectic.
\end{proposition}
\begin{proof}
The isomorphism $\overline{\alpha}$ is Poisson if and only if its differential $\overline{\alpha}_*$ maps the Poisson bivector of $\Zbar$ to the Poisson bivector of $\He$. It is sufficient to check this condition on the open dense symplectic leaf---that is, to check that the isomorphism of varieties
\[\alpha:\mathcal{Z}\longrightarrow\Hstdcirc\]
is a symplectomorphism when $\mathcal{Z}$ is equipped with the Poisson structure described in Section \ref{symplectic} and $\Hstdcirc$ is equipped with the Poisson structure coming from Hamiltonian reduction of $G\times_B\stdcirc$.

Consider again the moment map 
\begin{align*}
\mu_R:\,\,T^*G&\longrightarrow \mathfrak{g}\\
	(a,x)&\longmapsto x
	\end{align*}
for the right action of $G$ on $T^*G$. Since $\mu_R$ is $G$-equivariant and Poisson, we have an isomorphism of coisotropic varieties
\[T\times \mu_R^{-1}(f+\mathfrak{b})\cong\mu_R^{-1}(\stdcirc)\]
induced by the action map, where the coisotropic structure on $T$ is zero. Quotienting by the action of $B$ gives a symplectomorphism
\begin{equation}
\label{symplecto}
G\times_N(f+\mathfrak{b})\cong \mu_R^{-1}(f+\mathfrak{b})/N\cong \mu_R^{-1}(\stdcirc)/B\cong G\times_B\stdcirc.
\end{equation}

The universal centralizer $\mathcal{Z}$ is the Whittaker reduction of the left-hand side with respect to the left action of $N$, and the leaf $\Hstdcirc$ is the Whittaker reduction of the right-hand side with respect to $\muh$. Since the isomorphism $\alpha$ is induced from \eqref{symplecto} by taking Whittaker reduction of both sides, it is a symplectomorphism.
\end{proof}

\begin{theorem}
The Poisson structure on the space $G\times_B H$ is log-symplectic.
\end{theorem}
\begin{proof}
The moment map 
\[\muh:G\times_B H\longrightarrow \mathfrak{g}\]
is a Poisson morphism and $\Sl$ is a Poisson transversal in $\mathfrak{g}$. By Lemma \ref{transversals}, the preimage
\[\mu^{-1}_{H}(\Sl)\cong\He\]
is a Poisson transversal in $G\times_BH$, and by Proposition \ref{log-isom} it is log-symplectic. It remains to show that it intersects every symplectic leaf of $G\times_BH$---then the theorem will follow from Proposition \ref{log-transversals}.

The symplectic leaves of $G\times_BH$ are in bijection with the $B$-orbits on the quotient space $H/\mathfrak{n}$. If $\mathcal{O}\subset H/\mathfrak{n}$ is such an orbit, we write $\mathcal{O}+\mathfrak{n}$ for its preimage in $H$. Then the corresponding symplectic leaf is $G\times_B(\mathcal{O}+\mathfrak{n}).$

For each simple root $\alpha\in \Delta$, let $\check{\alpha}$ be the corresponding coroot and let $f_\alpha\in\mathfrak{g}_{-\alpha}$ be a fixed choice of negative simple root vector. Then the $B$-orbits on $H/\mathfrak{n}$ are indexed by the data
\begin{itemize}
\item a subset of simple roots $I\subset\Delta$
\item a collection of elements $\{h_{\alpha}\in\mathbb{C}\check{\alpha}\mid\alpha\not\in I\},$
\end{itemize}
in the sense that each orbit $\mathcal{O}\subset H/\mathfrak{n}$ contains a unique coset of the form
\[\left(\sum_{\alpha\in I}f_\alpha\right)+\left(\sum_{\alpha\not\in I}h_\alpha\right)+\mathfrak{n}.\]
Then $\mathcal{O}+\mathfrak{n}$ contains the element
\[y=\left(\sum_{\alpha\in I}f_\alpha\right)+\left(\sum_{\alpha\not\in I}h_\alpha\right)+e\in \mathfrak{b}^-+e\subset\mathfrak{g}^{\text{r}}.\]
Since $y$ is regular, there is some $g\in G$ such that $gy\in\Sl$. The point 
\[[g:y]\in G\times_BH\]
 lies both in the symplectic leaf corresponding to $\mathcal{O}$ and in the Poisson transversal $\muh^{-1}(\Sl)$.\qedhere\\
\end{proof}

\section{Relation to Coulomb branches}
\label{coulomb}

Let $\check{G}$ be the Langlands dual group of $G$, and denote by $\mathcal{K}=\mathbb{C}((t))$ the field of Laurent series and by $\mathcal{O}=\mathbb{C}[[t]]$ its ring of integers. The affine Grassmannian of $\check{G}$ is the ind-scheme
\[\gr\coloneqq \check{G}(\mathcal{K})/\check{G}(\mathcal{O}).\]
The $\check{G}(\mathcal{O})$-orbits on $\gr$ are finite-dimensional varieties indexed by $\Lambda^+$, the set of dominant characters of the maximal torus $T$ of $G$. They form a stratification of $\gr$, ordered by the standard partial order on the character lattice.

The equivariant homology space 
\[H^{\check{G}(\mathcal{O})}_\bullet(\gr)\]
has a ring structure given by the convolution product \cite[Section 2.7]{chr.gin:97}. It is a Poisson algebra \cite{bez.fin.mir:14} whose Poisson structure comes from the non-commutative one-parameter deformation 
\[H^{\check{G}(\mathcal{O})\rtimes\mathbb{C}^*}_\bullet(\gr).\]
This is an example of a \emph{Coulomb branch} in the sense of Nakajima \cite{nak:17}.

In \cite{bez.fin.mir:14} the authors construct an isomorphism of Poisson algebras
\begin{equation}
\label{bfm}
H^{\check{G}(\mathcal{O})}_\bullet(\gr)\cong\mathbb{C}[\mathcal{Z}].
\end{equation}
In this section we will explain, through the lens of this isomorphism, how to obtain the partial compactification $\Zbar$ directly from the Coulomb branch $H^{\check{G}(\mathcal{O})}_\bullet(\gr)$.

First notice that both sides of \eqref{bfm} have natural filtrations indexed by the lattice $\Lambda$ of characters of $T$. The filtration on the equivariant homology ring $H^{\check{G}(\mathcal{O})}_\bullet(\gr)$ is induced by the support in $\check{G}(\mathcal{O})$-orbit closures. The filtration on the coordinate ring $\mathbb{C}[\mathcal{Z}]$ is inherited through the surjection
\[\mathbb{C}[G]\otimes\mathbb{C}[\Sl]\cong\mathbb{C}[G\times \Sl]\twoheadrightarrow\mathbb{C}[\mathcal{Z}]\]
from the Peter-Weyl filtration on $\mathbb{C}[G]$.

\begin{proposition}
\label{bfm-explained}
The isomorphism \eqref{bfm} is an isomorphism of filtered algebras.
\end{proposition}
\begin{proof}
We recall an outline of the construction of \eqref{bfm}. Let $R$ be the set of roots of $\mathfrak{g}$ and $W$ the Weyl group. In \cite[Proposition 2.8]{bez.fin.mir:14}, the universal centralizer $\mathcal{Z}$ is identified with the spectrum of the Weyl group invariants of an affine blow-up of $T\times\mathfrak{t}$:
\[\mathbb{C}[\mathcal{Z}]\cong\mathbb{C}\left[T\times\mathfrak{t}, \frac{(t^\alpha-1)\times0}{1\times\alpha}\mid \alpha\in R\right]^W.\]
Here we write $\alpha\in R$ for the function on $\mathfrak{t}$ and $t^\alpha$ for the corresponding function on $T$. The right-hand side is filtered by dominant weights of $T$, and from the construction it is clear that this isomorphism is compatible with the filtrations on both sides.

Let $\check{T}$ be the maximal torus of $\check{G}$. Using the fixed-point localization theorem, it is shown in \cite[Section 6.3]{bez.fin.mir:14} that there is an isomorphism of localized $\mathbb{C}[\mathfrak{t}]^W$-modules
\begin{align*}
H^{\check{G}(\mathcal{O})}(\gr)_{\vert\mathfrak{t}^{\text{r}}/W} \cong H^{\check{T}(\mathcal{O})}(\grt)^W_{\vert\mathfrak{t}^{\text{r}}/W} 
												\cong \mathbb{C}[T\times \mathfrak{t}]^W_{\vert\mathfrak{t}^{\text{r}}/W}  
												\cong \mathbb{C}[\mathcal{Z}]_{\vert\mathfrak{t}^{\text{r}}/W}. 
\end{align*}
The authors then prove that this restricts to the desired isomorphism \eqref{bfm}. The filtrations we are interested in are compatible with the localization and with the first and third isomorphisms above. But it is clear that they also coincide under the second, which is induced by
\[H^{\check{T}(\mathcal{O})}(\grt) \cong \mathbb{C}\Lambda\otimes H^{\check{T}(\mathcal{O})}(\pt)
							\cong \mathbb{C}\Lambda\otimes \mathbb{C}[\mathfrak{t}]
							\cong \mathbb{C}[T \times\mathfrak{t}]. \qedhere\]
\end{proof}

Now, the wonderful compactification $\Gbar$ admits a Rees-type construction from the Peter-Weyl filtration on $\mathbb{C}[G]$ as follows. One considers the Rees algebra
\[\rees_\Lambda\mathbb{C}[G]=\bigoplus_{\lambda\in\Lambda}\mathbb{C}[G]_{\leq\lambda}t^\lambda\subset\mathbb{C}[G\times T],\]
where 
\[\mathbb{C}[G]_{\leq\lambda}=\bigoplus_{\mu\leq\lambda}V_\mu^*\otimes V_\mu\]
is a sum over all dominant weights $\mu$ which are less than or equal to $\lambda$ in the partial ordering on $\Lambda$. The wonderful compactification $\Gbar$ is obtained from $\rees_\Lambda\mathbb{C}[G]$ by a multi-proj construction \cite[Chapter 6]{bri.kum:05}: 
\[\Gbar\cong\proj\left(\rees_\Lambda\mathbb{C}[G]\right).\]
In other words, the \emph{Vinberg monoid}
\[V_G\coloneqq \spec\left(\rees_\Lambda\mathbb{C}[G]\right)\]
carries a natural action of $T$, and $\Gbar$ is the quotient of an open dense subset by this action.

An analogous procedure produces the partial compactification $\Zbar$. We take the Rees algebra of $\mathbb{C}[\Z]$ with respect to the filtration by weights of $T$ to obtain
\[\Zbar\cong\proj_{\Sl}\left(\rees_\Lambda\mathbb{C}[\Z]\right),\]
where the right-hand side is a multi-proj relative to the Kostant slice $\Sl$. In view of Proposition \ref{bfm-explained}, we have proved the following result:
\begin{proposition}
There is an isomorphism of projective $\Sl$-schemes $\displaystyle \Zbar\cong \proj_{\Sl}\left(\rees_\Lambda H^{\check{G}(\mathcal{O})}_\bullet(\gr)\right).$
\end{proposition}

\begin{remark}
This construction is similar to the approach suggested in \cite[Remark 3.7]{bra.fin.nak:19}.
\end{remark}

\bibliographystyle{plain}
\bibliography{draft}

\end{document}